\documentclass[12pt]{amsart}
\usepackage[]{fontenc}
\usepackage[latin9]{inputenc}
\usepackage[letterpaper]{geometry}
\usepackage{amsthm}
\usepackage{amstext}
\usepackage{setspace}
\usepackage{amssymb}
\makeatletter
\numberwithin{equation}{section} 
\numberwithin{figure}{section} 
\theoremstyle{plain}
\newtheorem{thm}{Theorem}[section]
  \theoremstyle{plain}
  \newtheorem{cor}[thm]{Corollary}
  \theoremstyle{definition}
  \newtheorem{problem}[thm]{Problem}
 \theoremstyle{definition}
  \newtheorem{example}[thm]{Example}
  \theoremstyle{remark}
  \newtheorem*{acknowledgement*}{Acknowledgement}
  \theoremstyle{plain}
  \newtheorem{lem}[thm]{Lemma}
  \theoremstyle{plain}
  \newtheorem{prop}[thm]{Proposition}
  \theoremstyle{remark}
  \newtheorem{claim}[thm]{Claim}

\DeclareMathOperator{\supp}{supp}

\makeatother

\begin{document}

\title{On notions of determinism in topological dynamics}

\author{Michael Hochman}
\begin{abstract}
We examine the relation between topological entropy, invertability,
and prediction in topological dynamics. We show that topological determinism
in the sense of Kami\'{n}sky Siemaszko and Szyma\'{n}ski imposes
no restriction on invariant measures except zero entropy. Also, we
develop a new method for relating topological determinism and zero
entropy, and apply it to obtain a multidimensional analog of this
theory. We examine prediction in symbolic dynamics and show that while
the condition that each past admit a unique future only occurs in
finite systems, the condition that each past have a bounded number
of future imposes no restriction on invariant measures except zero
entropy. Finally, we give a negative answer to a question of Eli Glasner
by constructing a zero-entropy system with a globally supported ergodic
measure in which every point has multiple preimages.
\end{abstract}
\maketitle

\section{Introduction}

There are several ways to define {}``determinism'' of a dynamical
system, all of which express the idea that the past determines the
future (and visa versa). In ergodic theory, a measure-preserving map
$T$ of a probability space $(X,\mathcal{B},\mu)$ is deterministic
if, for every measurable $f:X\rightarrow\mathbb{R}$ (or equivalently
every finite-valued $f$), the sequence $f(Tx),f(T^{2}x),\ldots$
determines $f(x)$ with probability one, that is, $f\in\sigma(Tf,T^{2}f,\ldots)$,
where $\sigma(\mathcal{F})$ is the $\sigma$-algebra generated by
$\mathcal{F}$. Another equivalent condition is that every factor
$(Y,\mathcal{C},\nu,S)$ of $(X,\mathcal{B},\mu,T)$ is essentially
invertible, i.e., there is an invariant set $Y_{0}\subseteq Y$ of
full measure such that $S|_{Y_{0}}$ is invertible. Yet another equivalent
condition which is widely used is that entropy vanish: $h(T,\mu)=0$.

In this work we examine the relations between prediction, invertability
and entropy in the category of topological dynamics, where by a topological
dynamical system $(X,T)$ we mean a continuous onto map $T:X\rightarrow X$
of compact metric space. One can find analogs of these three conditions,
but the relations between them are more complex. We present here several
results that underscore the independence of these notions, complementing
some of the recent works on the subject, e.g. \cite{NiteckiPrzytycki99,FiebigFiebigNitecki03,ChengNewhouse05}.

\subsection{\label{sub:Topological-predictability}Topological predictability}

Kami\'{n}ski, Siemaszko and Szyma\'{n}ski introduced in \cite{KaminskiSiemaszkoSzymanski03}
an interesting and natural notion of predictability and for topological
systems. A system $(X,T)$ is \emph{topologically predictable}%
\footnote{\emph{Kami}\'{n}\emph{ski et. al. use the term topological determinism,
but this seems to us confusing in the present context.}%
}, or TP,\emph{ }if for every continuous function $f\in C(X)$ we have
$f\in\left\langle 1,Tf,T^{2}f,\ldots\right\rangle $, where $\left\langle \mathcal{F}\right\rangle \subseteq C(X)$
denotes the closed algebra generated by a family $\mathcal{F}\subseteq C(X)$.
Kami\'{n}ski et. al. showed that $(X,T)$ is topologically predictable
if and only if every factor of $(X,T)$ is invertible, where a factor
is a system $(Y,S)$ and a continuous onto map $\pi:X\rightarrow Y$
such that $\pi T=S\pi$. 

One would like to understand what other dynamical implications topological
predictability has. In \cite{KaminskiSiemaszkoSzymanski2005} it was
shown that a TP systems have zero topological entropy (correcting
a gap, as the authors note, in their earlier proof from \cite{KaminskiSiemaszkoSzymanski03}),
but the converse to this is false. Indeed, every TP system on a totally
disconnected space is equicontinuous, whereas every zero entropy measure
can be realized as an invariant measure on a totally disconnected
space (and hence, for measures with irrational or continuous spectrum,
not TP).

Nonetheless, although {}``not TP'' seems to say little about the
invariant measures, TP is a rather strong condition, and one might
suppose it to impose restrictions on the measurable dynamics. In previous
work on the subject, the main tool used to establish that a system
is TP was the fact that, if every point in the product $(X\times X,T\times T)$
is forward recurrent, then $(X,T)$ is TP. Consequently, distal systems
and the pointwise rigid systems are TP; but no others were known. 

Our first result, which may be of independent interest, is that TP
imposes no restrictions on invariant measures except zero entropy.
\begin{thm}
\label{thm:TP-realization}For every zero-entropy, ergodic measure-preserving
system $(X,\mathcal{B},\mu,T)$ there is a topological system $(Y,S)$
and an invariant measure $\nu$ on $Y$ such that $(Y,\nu,S)\cong(X,\mathcal{B},\mu,T)$,
and, for every $y',y''$ in $Y$, the point $(y',y'')$ is forward
recurrent for $S\times S$. In particular, $(Y,S)$ is TP.
\end{thm}
This construction is related to the construction in B. Weiss \cite{Weiss98}.
For any zero entropy measure preserving system, that construction
produces, as a by-product, a topological model in which every pair
is two-sided recurrent in the product system. However, that is a far
weaker statement than forward recurrence. In fact, the realization
in \cite{Weiss98} is on a subshift, which is totally disconnected,
and one cannot hope that such a system will be TP (for then the action
would be equicontinuous, and the invariant measures would have pure
point spectrum). 

As a consequence of Theorem \ref{thm:TP-realization} one gets a new
functional characterization of the vanishing of entropy in a measure
preserving systems: 
\begin{cor}
\label{cor:entropy-characterization}A measure preserving system $(X,\mathcal{B},\mu,T)$
has entropy $0$ if and only if there exists a separable sub algebra
$\mathcal{A}\subseteq L^{\infty}(\mu)$ which separates points and
such that $f\in\left\langle 1,Tf,T^{2}f,\ldots\right\rangle $ for
every $f\in\mathcal{A}$.
\end{cor}
Next, we discuss the notion of TP for $\mathbb{Z}^{d}$ actions. Such
an action $\{T^{u}\}_{u\in\mathbb{Z}^{d}}$ of $\mathbb{Z}^{d}$ by
homeomorphisms on $X$ is topologically predictable (TP) if $f\in\left\langle 1,T^{u}f\,:\, u<0\right\rangle $
for every $f\in C(X)$; here $<$ is the lexicographical ordering
on $\mathbb{Z}^{d}$. One can also work with other orderings, e.g.
lexicographic orderings with respect to other coordinate systems.
One may ask whether this notion is independent of the generators (the
lexicographic ordering certainly is not). It is not; even in dimension
1, the property TP depends on the generator, i.e. TP for $T$ does
not imply it for $T^{-1}$. Thus TP is a property of a group action
and a given set of generators.

The proof in \cite{KaminskiSiemaszkoSzymanski2005,KaminskiSiemaszkoSzymanski03}
that TP implies 0 entropy for a single transformation used the non-trivial
theory of extreme partitions and entropy pairs. In section \ref{sub:Zd-TP}
we give a new and direct argument for this implication, which is somewhat
more transparent. Furthermore, our proof can be used to generalize
the result to actions of $\mathbb{Z}^{d}$. 
\begin{thm}
\label{thm:Zd-determinism}For a $\mathbb{Z}^{d}$-action, TP implies
zero topological entropy. 
\end{thm}
There is a rather complete theory of entropy, developed by Ornstein
and Weiss, for actions of amenable groups on probability spaces. One
feature which is absent from the general theory (and which we utilized
for $\mathbb{Z}$ and $\mathbb{Z}^{d}$ actions) is a good notion
of the {}``past'' of an action, and the ability to represent the
entropy of a partition as a conditional entropy of the partition with
respect to the {}``past''. However by analogy to the abelian case
the following question is natural:
\begin{problem}
\noindent Suppose an infinite discrete amenable group $G$ acts by
homeomorphisms on $X$. Let $S\subseteq G$ be a sub semigroup not
containing the unit of $G$, and such that $S\cup S^{-1}$ generates
$G$. Suppose that for every $f\in C(X)$ we have $f\in\left\langle 1,sf\,:\, s\in S\right\rangle $.
Does this imply that $h(X,G)=0$? 
\end{problem}

\subsection{\label{sub:Prediction-for-symbolic}Prediction for symbolic systems}

Let $\Sigma$ be a finite set of symbols and consider the space $\Sigma^{\mathbb{Z}}$
of bi-infinite sequences over $\Sigma$. Denote by $\sigma:\Sigma^{\mathbb{Z}}\rightarrow\Sigma^{\mathbb{Z}}$
the shift map. A symbolic system is a closed, non-empty, $\sigma$-invariant
subset of $\Sigma^{\mathbb{Z}}$.

Let $X\subseteq\Sigma^{\mathbb{Z}}$ be a subshift and let $x^{-}\in\Sigma^{-\mathbb{N}}$,
where $\mathbb{N}=\{1,2,3\ldots\}$; for $x\in\Sigma^{\mathbb{Z}}$
we also write $x^{-}=x|_{-\mathbb{N}.}$. A finite or infinite sequence
$x^{+}\in\cup_{0\leq n\leq\infty}\Sigma^{n}$ is an admissible extension
of $x^{-}$ (with respect to $X$) if the concatenation $x^{-}x^{+}$
is in $X$. If $h(X)=0$ then $h(\mu)=0$ for every invariant measure
$\mu$ on $X$, and so there is a set of points $X_{0}\subseteq X$,
having full measure with respect to every invariant measure, such
that $x^{-}$ has a unique extension for every $x\in X_{0}$; that
is, if $y\in X_{0}$ is another point, then $y^{-}=x^{-}$ implies
$x=y$. A natural question is whether this can occur for every $x,y\in X$.
The answer is no: in fact, it is well known that the only subshifts
for which every admissible past $x^{-}$ admits a unique continuation
are finite unions of periodic orbits (we give a proof in lemma \ref{lem:no-perfect-prediction}).

However, there do exist subshifts where each $x^{-}\in\Sigma^{-\mathbb{N}}$
has only finitely many extensions; the best known are probably the
Sturmian subshifts. Such subshifts must have zero entropy. It turns
out that such systems are not uncommon, and that entropy is again
the only restriction to the dynamics of their invariant measures:
\begin{thm}
\label{thm:symbolic-realization}Every ergodic measure-preserving
system with entropy zero is isomorphic to a shift-invariant Borel
measure on a uniquely ergodic subshift $X\subseteq\{0,1\}^{\mathbb{Z}}$
with the property that every $x^{-}\in\{0,1\}^{-\mathbb{N}}$ has
at most two infinite extensions.
\end{thm}
This may be viewed as a sharpening of the Jewett-Krieger generator
theorem, which states that every measure-preserving system with finite
entropy $h$ can be realized as the unique invariant measure on a
uniquely ergodic subshift on $k$ symbols, provided $\log k>h$. In
zero entropy, one cannot use less than 2 symbols. This theorem says
that one can do the next best thing.

\subsection{\label{sub:Non-invertibility-and-entropy}Non-invertability and entropy }

Consider a symbolic system $X\subseteq\Sigma^{\mathbb{N}}$ (note
that we now have a one-sided shift), and an invariant probability
measure $\mu$ on $X$. Recall that, since the partition of $X$ according
to the first symbol generates the $\sigma$-algebra, the entropy $h(\mu)$
is the average of the entropy of the conditional measures, given $x$,
induced on the preimage set $\sigma^{-1}(x)$. Thus if $h(\mu)>0$
then with positive probability $\sigma^{-1}(x)$ is not concentrated
on a single point, and consequently there is a large set of points
in $X$ with multiple preimages. It is therefore natural to ask what
{}``degree'' of non-invertability is necessary to guarantee positive
entropy.

One plausible condition is that each point have multiple preimages;
we call such a system \emph{everywhere non-invertible}. Indeed, for
subshifts this is enough to imply positive entropy, because, for symbolic
systems, everywhere non-invertability implies a stronger condition:
the preimage of every point has diameter $>\delta$ for some $\delta>0$.
Whenever this condition is satisfied we say that the system has \emph{no
small preimages. }An easy argument shows that a map with no small
preimages has entropy at least $\log2$ (see proposition \ref{pro:large-preimages-implies-entropy}
below). 

Everywhere non-invertability does not guarantee positive entropy in
general, though in some special cases it does, e.g. maps of the interval
\cite{Bobok02}. One would hope to find additional hypotheses which,
together with everywhere non-invertability, imply positive entropy.
One candidate is the presence of a globally supported ergodic measure.
In an everywhere non invertible system there is always an open set
of points whose preimages have diameter which is bounded below by
some positive constant, and when there is a globally supported ergodic
measure, almost every orbit spends a positive fraction of its time
in this set. One would hope to use this fact to construct many well-separated
orbits. Eli Glasner has raised the question of whether this hypothesis
indeed implies positive entropy. We show that it does not:
\begin{example}
\label{exa:non-invertible-example}There exists a zero entropy, everywhere
non-invertible systems with a globally supported ergodic measure.
\end{example}
For an integer $k>0$, we say that a system $(X,T)$ is at least \emph{$k$-to-one}
if the preimage set of every point is of size at least $k$. J. Bobok
has shown that if a map of the circle (or the interval) is $k$-to-one,
then $h(T)\geq\log k$, and has asked if this holds in general, at
least under the assumption that there are no small preimages. We can
give a negative answer to this:
\begin{example}
\label{exa:non-invertible-example-2}There exists an infinite-to-one
system $(X,T)$ with no small preimages, and which supports a global
ergodic invariant measure, but $h(X,T)=\log2$.
\end{example}
There seems to be no obstruction in our examples to making the measures
weakly mixing, and possibly strong mixing, but we do not pursue this
here.

The question remains whether such examples exist for a continuous
map on a manifold. For smooth maps they do not, see \cite{Bobok05}. 
\begin{acknowledgement*}
This work was done during the author's graduate studies. I would like
to thank Benjamin Weiss for his constant encouragement and for raising
some of the questions addressed here.
\end{acknowledgement*}

\section{\label{sec:Preliminaries}Notation}

We will use freely standard facts about topological dynamics and entropy
which can be found e.g. in \cite{Walters82}. This section contains
some further notation for dealing with sequence spaces.

Let $\Sigma$ be a set and write $\Sigma^{*}$ for the set of all
finite words over $\Sigma$. The $i$-th letter of a word $a\in\Sigma^{*}$
is denoted by $a(i)$. If $a=a(1)a(2)\ldots a(k)$ then $k$ is the
length of $a$ and is denoted by $\ell(a)$. We denote concatenation
the of words $a,b\in\Sigma^{*}$ by $ab$. 

Similarly, we define the spaces of one-sided sequences $\Sigma^{\mathbb{N}},\Sigma^{-\mathbb{N}}$
(we use the convention $\mathbb{N}=\{1,2,3,\ldots\}$) and of two-sided
sequences, $\Sigma^{\mathbb{Z}}$. If a topology is given on $\Sigma$
these sequence spaces carry the product topology; for finite $\Sigma$
we take the discrete topology for $\Sigma$. We denote by $\sigma$
the shift map on both these spaces which is defined by the formula
$(\sigma(x))(i)=x(i+1)$; this map restricted to $\Sigma^{\mathbb{N}}$and
$\Sigma^{\mathbb{Z}}$ is continuous and onto, and is a homeomorphism
in the two-sided case. In the one sided case the preimage set of every
point is identified with $\Sigma$. We also define the shift on $\Sigma^{*}$
in the obvious way, by\[
\sigma(x(1)x(2)\ldots x(k))=x(2)x(3)\ldots x(k)\]
(note that $\sigma^{n}(ab)=(\sigma^{n}a)b$ if $n\leq\ell(a)$ but
is equal to $\sigma^{n-\ell(a)}(b)$ if $\ell(a)<n\leq\ell(a)+\ell(b)$.
Otherwise it is the empty word). When concatenating infinite sequences,
we adopt the convention that, if $x\in\Sigma^{-\mathbb{N}}$ and $y\in\Sigma^{\mathbb{N}}$,
then $xy\in\Sigma^{\mathbb{Z}}$ is the sequence $z$ with $z(i)=x(i)$
for $i<0$ and $z(i)=y(i+1)$ for $i\geq0$ (note that $0\notin\mathbb{N}$,
which is the reason for this shift of $y$).

For a word $x$ (finite or infinite), if $x=ab$ then $a$ is called
a \emph{front segment} of $x$ (if $\ell(a)=k$ then $a$ is a front
$k$-segment of $x$), and $b$ a \emph{back segment} of $x$. For
$a,b\in\Sigma^{*}$ we say that $a$ is a subword of $b$ at index
$i$ if $i\leq\ell(b)-\ell(a)+1$ and $a(j)=b(i+j)$ for $j=1,\ldots,\ell(a)$.
The index $i$ is called the \emph{alignment} of $a$ in $b$. If
such an $i$ exists we say that $a$ appears in $b$, or that it is
a subword of $b$. 

We denote by $[i;j]$ the segment of consecutive integers $[i,j]\cap\mathbb{Z}$,
and denote by $x|_{[i;j]}=x(i)x(i+1)\ldots x(j)$ the subword of $x$
determined by $[i;j]$, provided $x$ is long enough for this to make
sense.

All measures are assumed to be Borel probability measures.

\section{\label{sec:Topological-predictability}Topological predictability}

\subsection{\label{sub:Realizing-TP-systems}Realization of measures on TP systems}

A topologically predictable system has zero topological entropy, and
therefore, by the variational principle, every invariant measure on
it has entropy zero. In this section we prove Theorem \ref{thm:TP-realization},
showing that this is the only restriction on invariant measures. The
construction is rather technical. We remark that this section is not
used in the sequel.

A point $x$ in a dynamical system $(X,T)$ is forward recurrent if
$T^{n(k)}x\rightarrow x$ for some sequence of times $n(k)\rightarrow\infty$.
Note that if every point in a system is forward recurrent then every
closed subset $A\subseteq X$ which is forward invariant, i.e. $TA\subseteq A$,
is invariant, i.e. $T^{-1}A=TA=A$. 

In order to construct a TP system supporting a given measure we shall
construct an isomorphic measure on a topological system $(X,T)$ for
which every point in $X\times X$ is forward recurrent. Indeed, by
the remark above, this implies that every forward invariant, closed
equivalence $R\subseteq X\times X$ is also invariant under $T^{-1}$,
and this is equivalent to the property that every factor is invertible,
so $(X,T)$ is topologically predictable \cite{KaminskiSiemaszkoSzymanski03}.
Our construction cannot be symbolic since since infinite symbolic
systems always contain forward-asymptotic pairs. We shall instead
construct a connected subshift of $[0,1]^{\mathbb{N}}$.

Let $(X,\mathcal{B},\mu,T)$ be a measure-preserving system with zero
entropy. We wish to construct a space $Y$ and homeomorphism $S:Y\rightarrow Y$
for which every pair is forward recurrent and which supports a measure
isomorphic to $(X,\mathcal{B},\mu,T)$. 

For the construction we may assume, by e.g. \cite{Weiss98}, that
$T$ is a minimal, topologically weak mixing, strictly ergodic homeomorphism
of a totally disconnected metric space $X$, and that there exists
a clopen generator for $T$. 

Given a measurable function $f:X\rightarrow[0,1]$, let $f^{(m)}:X\rightarrow[0,1]^{m}$
denote the function $x\mapsto(f(x),f(Tx),\ldots,f(T^{m-1}x))$, and
similarly let $f^{(\infty)}:X\rightarrow[0,1]^{\mathbb{N}}$ denote
the map $x\mapsto(f(x),f(Tx),f(T^{2}x),\ldots)$. We use the notation
$\left\Vert a\right\Vert _{\infty}=\sup|a(i)|$ for $a\in\mathbb{R}^{m}$
or $a\in\mathbb{R}^{\mathbb{N}}$. 

For integers $m,r$, we say that $f$ is $(m,r)$-good if there is
a subset $X_{f,m,r}\subseteq X$ of full measure such that, for every
$x',x''\in X_{f,m,r}$, there is an integer $0<k<r$ (which may depend
on $x',x''$) satisfying \begin{eqnarray*}
\left\Vert f^{(m)}(x')-f^{(m)}(T^{k}x')\right\Vert _{\infty} & < & \frac{1}{m}\\
\left\Vert f^{(m)}(x'')-f^{(m)}(T^{k}x'')\right\Vert _{\infty} & < & \frac{1}{m}.\end{eqnarray*}
Suppose that $f$ is $(m,r(m))$-good for some sequence $r(m)$. Setting
$X_{0}=\cap_{m=1}^{\infty}X_{f,m,r(m)}$, the relation above holds
for every $x',x''\in X_{0}$ and all $m\in\mathbb{N}$. If we set
$\nu=f^{(\infty)}\mu$ and $Y=\supp\nu\subseteq[0,1]^{\mathbb{N}}$,
it follows that each pair of points in $Y$ is forward recurrent for
the shift $\sigma$. Also, $\nu$ is shift invariant on $(Y,\sigma)$,
and $f^{(\infty)}$ is a factor map from $X$ to $Y$, and if the
partition induced by $f$ on $X$ generates for $T$ then this is
an isomorphism. Thus the theorem will follow once we construct a function
$f$ as above.

We construct $f$ by approximation. More specifically, we define a
sequence of functions $f_{n}:X\rightarrow[0,1]$ and integers $r(n)$
such that $f_{n}$ is $(m,r(m))$ good for each $m\leq n$. The sequence
$f_{n}$ will be constructed so that it converges a.e. to a function
$f$, which is clearly $(m,r(m))$ good for $m\in\mathbb{N}$. Also,
each $f_{n}$ will generate for $T$ and we will guarantee that $f$
generates by controlling the speed of convergence of $f_{n}$ to $f$.
The $f_{n}$'s will be continuous and each will take on only finitely
many values, so we may identify them with finite partitions $P_{n}$
of $X$ into clopen sets, where $f_{n}(x)=i$ if and only if $x$
is in the partition element of $P_{n}$ indexed by $i$ (we allow
$i$ to take non-integer values). 

The construction proceeds by induction. Our induction hypothesis will
be that we are given a function $f_{n}$ arising from a finite clopen
generating partition $P_{n}$, and integers $r(1),\ldots,r(n)$, such
that $f_{n}$ is $(m,r(m))$-good for $m=1,\ldots,n$. For any $\varepsilon$,
we will show how to define $f_{n+1}$ and $r(n+1)$ satisfying the
same condition with $n+1$ in place of $n$, and such that \[
\mu(x\in X\,:\, f_{n}(x)\neq f_{n+1}(x))<\varepsilon.\]
By choosing $\varepsilon=\varepsilon(n)$ to decrease rapidly enough
this last condition guarantees that $f_{n}\rightarrow f$ almost surely,
and that $f$ generates for $T$ (see e.g. \cite{Shields73}). 

Suppose then that we are given $f_{n}$, $r(1),\ldots,r(n)$ and $\varepsilon>0$
as above. First, note that the properties of these objects are completely
determined by the itineraries of length $r(n)+n$ associated under
$f_{n}$ to points in $X$, i.e. by the image of $f_{n}^{(r(n)+n)}$.
The following lemma, whose proof we omit, says that the desired properties
of the blocks continue to hold if we modify itineraries in a sufficiently
slow way:
\begin{lem}
\label{lem:perturbed-blocks}For $f_{n},P_{n},r(1),\ldots,r(n)$ as
above, there is a number $0<\rho<\frac{1}{n+1}$ with the following
property. Suppose $y',y''\in[0,1]^{r(n)+n}$ are blocks appearing
in $f_{n}^{(\infty)}(X)$ and $\alpha',\alpha''\in[0,1]^{r(n)+n}$
have the property that $|\alpha'(i)-\alpha'(i+1)|<\rho$ and $|\alpha''(i)-\alpha''(i+1)|<\rho$
for all $1\leq i\leq r(n)+n-1$. Define $z',z''\in[0,1]^{r(n)+n}$
by $z'(i)=\alpha'(i)\cdot y'(i)$ and $z''(i)=\alpha''(i)\cdot y''(i)$.
Then there exists $0<k\leq r(m)$ with $|z'(i)-z'(i+k)|<\frac{1}{m}$
and $|z''(i)-z''(i+k)|<\frac{1}{m}$ for $i=1,2,\ldots,n$. 
\end{lem}
Let $Y\subseteq[0,1]^{\mathbb{N}}$ be the symbolic subshift defined
by the property that every block of length $r(n)+n$ in $Y$ appears
in $f_{n}^{(\infty)}(X)$. Note that $Y$ is a shift of finite type
and is irreducible because $X$ is topologically mixing. In particular,
there is an integer $D$ such that given two blocks $a,c$ appearing
in $Y$, there is a block $b_{k}$ for every $k\geq D$ such that
$ab_{k}c$ appears in $Y$. We can also fix a block $a^{*}$ appearing
in $Y$ which contains a copy of every $n$-block in $Y$. Increasing
$D$ or lengthening $a^{*}$ if necessary, so may assume that $D>1/\varepsilon$
and that $a^{*}$ is of length $D$. 

We need the following, which is a specialized version of lemma 2 from
\cite{Weiss98}:
\begin{lem}
\label{lem:weiss-blocks}There exists $\delta>0$ and $T_{0}\in\mathbb{N}$
such that, for all $T\geq T_{0}$, there is a family $I$ of subsets
of $\{0,\ldots,T-1\}$ satisfying
\begin{enumerate}
\item $|I|\geq2^{\delta T}$,
\item For $A\in I$ and distinct $u,v\in A$, we have $|u-v|\geq\frac{10D}{\varepsilon}$,
\item For each $A,B\in I$ and $k\leq\frac{9T}{10}$, we have $A\cap(B+k)\neq0$.
\end{enumerate}
\end{lem}
We use the lemma in conjunction with the following simple fact:
\begin{lem}
\label{lem:indep-blocks}Fix $T$ and let $A,B\subseteq\{0,1,\ldots,T\}$
satisfy the three conditions of the previous lemma. Fix $0\leq k\leq\frac{9T}{10}-n$,
and let $z',z''\in[0,1]^{\mathbb{N}}$ such that $a^{*}$ appears
in $z'$ at each index $i\in A$ and in $z''$ at each index $j\in B+k$.
Then for every pair $a,b$ of $n$-blocks from $Y$, there is an index
$u$ such that $a$ appears in $z'$ at $u$, and $b$ appears in
$z''$ at $u$.
\end{lem}
Let $\rho,\delta,T_{0}$ be as in the preceding lemmas. Since $(X,T,\mu)$
has zero topological entropy, it follows that we can choose an integer
$H\geq\frac{10}{\rho\varepsilon}T_{0}$ and large enough so that $2^{\delta(\varepsilon\rho/10)H}$
is greater than the number of $(P_{n},H)$-names in $X$. We fix such
an integer $H$ and construct an Alpern tower \cite{EigenPrasad1997}
over some clopen set $B\subseteq X$, with columns of heights $H$
and $H+1$. This means that every point in $B$ returns to $B$ for
the first time after either $H$ or $H+1$ applications of $T$. The
$i$-th level of the tower is the set of points $T^{i}B\setminus B$,
and the disjoint union of these levels for $0\leq i\leq H+1$ is all
of $X$. The last property can be obtained because $(X,T)$ is minimal.
This is a standard modification of the construction of Alpern towers:
one begins the construction with a clopen set, and notes that, due
of minimality, all points eventually return to it. 

Purify the columns according to $P_{n}$, and let $B_{1}\ldots B_{N}$
be the bases of the purified columns. Thus, $\{B_{1}\ldots,B_{N}\}$
is a clopen partition of $B$ which refines the partition according
to return time, and, if $h(i)$ denotes the height of the column over
$B_{i}$, then all $x\in B_{i}$ have the same $P_{n}$-itinerary
up to time $h(i)$, and these itineraries are distinct for different
$i$. Note that the $P_{n}$-name of each column appears in $Y$.

Divide each column into $\frac{10}{\varepsilon\rho}$ blocks of length
$\frac{\varepsilon\rho}{10}H$ (which we assume for convenience is
an integer), and possibly an additional level in those columns which
are of height $H+1$. We proceed to modify $P_{n}$ as follows.
\begin{itemize}
\item In each column, re-name the bottom $1+\frac{1}{\rho}$ blocks so that
they are identical, and similarly for the top $1+\frac{1}{\rho}$
blocks; and do so in such a way that the name of the entire column
is admissible for $Y$. This can be done because $\frac{\rho\varepsilon}{10}H$,
the length of each block, is much larger than $D$. Notice that by
choice of $\rho$, the first and last $n+1$ blocks in each column
are identical.
\item To each block, except the top and bottom $n$ blocks of each column,
assign a distinct set $A\subseteq\{0,\ldots,\frac{\varepsilon\rho}{10}H-1\}$
such that $|u-v|\geq\frac{10D}{\varepsilon}$ for distinct $u,v\in A$,
and if $A,B$ are assigned to distinct blocks and $\frac{1}{10}\cdot\frac{\varepsilon\rho}{10}H\leq k\leq\frac{9}{10}\cdot\frac{\varepsilon\rho}{10}H$
then $A\cap(B+k)\neq\emptyset$. We can do this by the choice of $H$
and the lemma. To the bottom $n$ blocks in each column assign the
same set $A$ which is assigned to the $n+1$-st block of that column,
and similarly to the top $n$ blocks assign the same set which is
assigned to the $n+1$-th block from the top. We have thus assigned
a set to each block.
\item For a block $b$ appearing in one of the columns and the set $A$
associated to it, we modify $b$ as follows. For convenience, in this
paragraph we renumber the coordinates of $b$ from $0$ to $\frac{10}{\varepsilon\rho}-1$,
no matter where in the column $b$ actually appears. For each $i\in A$
we replace the block of length $D$ in $b$ starting at $i$ with
the block $a^{*}$. Next, modify the symbols from $i-D$ to $i-1$
and from $i+D$ to $i+2D-1$ in such a way that the entire block from
$i-2D$ to $i+3D$ appears in $Y$; we can do this by the definition
of $D$. All in all, we have changed $b$ from index $i-D$ to index
$i+2D-1$. Because of the distance between successive elements of
$A$, these changes for different $i\in A$ occur at different places
in $b$ and the changes do not interfere with each other.

Note that the bottom $n+1$ blocks of each column are still identical,
as are the $n+1$ top blocks.

Denote by $\tilde{P}_{n+1}$ the partition obtained so far, and by
$\tilde{f}_{n+1}$ the corresponding function.

\item If $b_{1},b_{2},\ldots,b_{1/\rho}$ are the bottom $\frac{1}{\rho}$
blocks of some column, replace $b_{k}$ with $(k-1)\rho\cdot b_{k}$,
where $\alpha\cdot b_{i}$ is the block obtained by multiplying each
coordinate of $b_{i}$ by $\alpha$. Similarly, if $c_{1},c_{2},\ldots,c_{1/\rho}$
are the top $n$ blocks of a column replace $c_{k}$ with $(1/\rho-k)\rho c_{k}$.
\item For columns of height $H+1$, replace the top symbol with $0$.
\item Perturb the first symbol of each column by less than $\varepsilon$
in a way that the name of each column is unique.
\end{itemize}
Let $f_{n+1}$ be the functions defined by the revised partition;
we claim that it has the desired properties for some integer $r(n+1)$.

We first estimate the measure of points on which $f_{n}$ and $f_{n+1}$
differ. It suffices to show that in each column the fraction of levels
modified is less than $\varepsilon$. The change to the top and bottom
$\frac{1}{\rho}$ blocks amounts to $\frac{2}{\rho}$ blocks out of
$\frac{10}{\varepsilon\rho}$, which is $\frac{\varepsilon}{5}$ of
the levels. Consider now the intermediate levels. Since in the sets
$A$ associated to the blocks the distance between elements is at
least $\frac{10D}{\varepsilon}$, and each element causes a change
of $3D$ symbols to its block, here too we have caused a change to
at most a $\frac{3\varepsilon}{10}$-fraction of the levels. The change
to the top symbol of columns of height $H+1$ amounts to less than
$\frac{1}{H}$ of the space. Thus we have indeed modified $f_{n}$
on a set of measure less than $\varepsilon$. 

We now show that we can choose $r(n+1)$ so that $f_{n+1}$ is $(m,r(m))$-good
for each $m\leq n+1$. Note that every block in $f_{n+1}^{(\infty)}(X)$
of length $r(n)+n$ is of the form described in lemma \ref{lem:perturbed-blocks},
so for $m\leq n$ the conclusion follows immediately from that lemma.

We must show that $f_{n+1}$ is $(n+1,r)$-good for some $r$. Let
$x',x''\in X$. We must show that there is a $k$ of bounded size
such that $\left\Vert f_{n+1}^{(n+1)}(x')-f_{n+1}^{(n+1)}(\sigma^{k}x')\right\Vert _{\infty}<\frac{1}{n+1}$
and $\left\Vert f_{n+1}^{(n+1)}(x'')-f_{n+1}^{(n+1)}(\sigma^{k}x'')\right\Vert _{\infty}<\frac{1}{n+1}$.
Denote $y'=\widetilde{f}_{n+1}(x')$ and $y''=\widetilde{f}_{n+1}(x'')$,
and also $z'=f_{n+1}^{(n+1)}(x')$ and $z''=f_{n+1}^{(n+1)}(x'')$.
We distinguish several cases.

\textbf{Case 1}. Both $x',x''$ are in the top block or level $H+1$
of their respective columns. Then the first $\frac{10}{\rho\varepsilon}$
symbols of $f_{n+1}^{(\infty)}(x'),f_{n+1}^{(\infty)}(x'')$ are $0$,
and the conclusion holds for $k=1$.

\textbf{Case 2}. Exactly one of the points, say $x'$, is in the top
block or level $H+1$ of its column, so the first $\frac{10}{\rho\varepsilon}$
symbols of $f_{n+1}^{(\infty)}(x')$ are $0$. Note that in $y''=\widetilde{f}_{n+1}(x'')$
the block $a^{*}$ appears somewhere between index $1$ and $\frac{10}{\rho\varepsilon}$,
hence there is a $0<k\leq\frac{10}{\rho\varepsilon}$ with $\left\Vert \tilde{f}_{n+1}^{(n+1)}(x'')-\tilde{f}_{n+1}^{(n+1)}(\sigma^{k}x'')\right\Vert _{\infty}=0$.
If we replace $\tilde{f}_{n+1}^{(n+1)}$ with $f_{n+1}^{(n+1)}$ the
left hand side changes by at most $\rho$ and we get\[
\left\Vert f_{n+1}^{(n+1)}(x'')-f_{n+1}^{(n+1)}(\sigma^{k}x'')\right\Vert _{\infty}<\rho.\]
On the other hand, $\left\Vert f_{n+1}^{(n+1)}(x')-f_{n+1}^{(n+1)}(S^{k}x')\right\Vert _{\infty}=0$
because the first $\frac{10}{r\varepsilon}$ symbols of the itinerary
of $x'$ are $0$; as desired.

\textbf{Case 3}. $x',x''$ are in different columns or the same column
but at least $\frac{1}{9}\cdot\frac{10}{\varepsilon\rho}$ levels
apart, and neither is in the top block or top level. By looking at
the blocks to which $x',x''$ belong and to the next block, by lemma
\ref{lem:indep-blocks} we see that for every pair of $n+1$-blocks,
and in particular the one appearing at the start of the itineraries
of $x',x''$, there is a $k$ in the range we want such that these
blocks appear again in the $\tilde{f}_{n+1}$ itinerary of both $x'$
and $x''$ at index $k$. As in case 2, this gives the conclusion
for the $f_{n+1}$ itinerary because the change from $\tilde{f}_{n+1}$
to $\tilde{f}_{n+1}$ is {}``too slow'' to affect the inequality
very much.

\textbf{Case 4}. $x',x''$ belong to the same column and are within
$\frac{1}{9}\cdot\frac{10}{\varepsilon\rho}$ levels of each other.
If they are in one of the bottom $\frac{1}{\rho}$ levels then we
are done by the periodicity of these blocks (again, there is some
slow {}``drift'' which does not affect us). Otherwise, the initial
$n+1$-block of both itineraries belongs to $Y$. We claim that there
is an $M$ such that either for some $0<i<M$ the points $T^{i}x',T^{i}x''$
belong to different columns but not to the top or bottom $\frac{1}{\rho}$
blocks of those columns, or else there exists a $k<M$ as desired.
This suffices because in the former case we can argue as in case 3,
and deduce that as $k$ ranges over the $1,\ldots,M+\frac{10}{\rho\varepsilon}$,
every pair of $n+1$-blocks from $Y$ appears at index $k$ in the
$f_{n+1}$-itineraries of $x',x''$. This gives the conclusion we
want.

It remains to show that there is such an $M$. This follows from the
fact that $f_{n+1}^{(\infty)}(X)$ is a minimal symbolic system. Indeed,
suppose the contrary. Then for every $M$ there exist points $x'_{M},x''_{M}\in X$
such that whenever $1\leq i\leq M$ and $T^{i}x'_{M},T^{i}x''_{M}$
are in different columns it is because they are within $\frac{10}{\varepsilon\rho}$
of the top or bottom of a column, and also the initial $n+1$-blocks
of the itineraries of $x',x''$ do not appear again together before
time $M$. We may assume that $x'_{M}\rightarrow x'$ and $x''_{M}\rightarrow x''$.
Now $x',x''$ have these properties as well, for all $M$. Assuming
as we may that $x'$ is above $x''$ in the column they belong to,
it follows that the itinerary of $x'$ is a shift of the itinerary
of $x''$, so the pair $(f_{n+1}^{(n+1)}(x'),f_{n+1}^{(n+1)}(x''))\in f_{n+1}^{(n+1)}(X)$
is of the form $(y,T^{r}y)$ for some $r\leq\frac{1}{9}\cdot\frac{10}{\rho\varepsilon}$.
But since $f_{n+1}^{(\infty)}(X)$ is minimal this point must be recurrent,
a contradiction. This completes the proof of theorem \ref{thm:TP-realization}.

Notice that the construction has introduced a fixed point $000\ldots$
in the resulting subshift. We do not know if this can be avoided;
more specifically, we do not know if the subshift can be made to be
minimal.

\subsection{\label{sub:Zd-TP}\label{sub:continuous-partitions}Partitions derived
from continuous functions and predictable $\mathbb{Z}^{d}$ actions}

In this section we prove a purely measure-theoretic and topological
lemma which involves no dynamics. Let $X$ be a normal topological
space and $\mu$ a regular probability measure on the Borel $\sigma$-algebra
of $X$. The entropy and conditional entropy of finite and countable
partitions is defined as usual \cite{Walters82}. For finite or countable
measurable partitions $\mathcal{P}=(P_{1},P_{2},\ldots)$ and $\mathcal{Q}=(Q_{1},Q_{2},\ldots)$
of $X$ with finite entropy, the Rohlin metric is defined by \[
d(\mathcal{P},\mathcal{Q})=H(\mathcal{P}|\mathcal{Q})+H(\mathcal{Q}|\mathcal{P})\]
This metric has the property that if $\mathcal{P}=(P_{1},P_{2},\ldots)$
and we define $\mathcal{P}^{(n)}=(P_{1},\ldots,P_{n},\cup_{k=n+1}^{\infty}P_{k})$,
then $\mathcal{P}^{(n)}\rightarrow\mathcal{P}$ in $d$.

We say that a partition $\mathcal{P}$ is \emph{continuous} if there
is continuous function $f\in C(X)$ which is constant almost surely
on each atom of $P_{i}$. Equivalently, $\mathcal{P}$ agrees with
the partition of $X$ into level sets of some $f\in C(X)$, up to
measure zero.
\begin{prop}
\label{pro:continuous-partitions-are-dense}The continuous partitions
are dense with respect to the Rohlin metric in the space of finite-entropy
countable partitions.\end{prop}
\begin{proof}
The proof is a variation on Urisohn's lemma which states that given
two closed disjoint sets $C_{0},C_{1}$ in a normal space, there is
a continuous function $0\leq f\leq1$ such that $f^{-1}(0)=C_{0}$
and $f^{-1}(1)=C_{1}$. 

Let $\mathbb{D}\subseteq\mathbb{Q}\cap[0,1]$ denote the dyadic rationals.
Let $\mathcal{P}=(P_{0},P_{1})$ be a partition into two sets and
let $\varepsilon>0$. We construct a continuous function $f:X\rightarrow[0,1]$
with $\mu(\cup_{r\in\mathbb{D}}f^{-1}(r))=1$ such that the countable
partition $\mathcal{Q}=\{f^{-1}(r)\,:\, r\in\mathbb{D}\}$ satisfies
$d(\mathcal{P},\mathcal{Q})<\varepsilon$. The proof in case $\mathcal{P}$
has more than two atoms is similar; this is sufficient, because the
finite partitions are dense in the Rohlin metric.

We construct a family of open sets $\{U_{r}\}_{r\in\mathbb{D}}$ with
$\overline{U_{r}}\subseteq U_{s}$ for $r\leq s$ and with $\mu(\partial U_{r})=0$.
We will also define closed disjoint sets $(C_{r})_{r\in\mathbb{D}}$
such that $C_{s}\subseteq U_{t}\setminus U_{r}$ for all $r<s<t$,
and $\mu(\cup C_{r})=1$. We will then define $f$ by \[
f(x)=\inf(\{1\}\cup\{r\,:\, x\in U_{r}\})\]
This defines a continuous function with $f|_{C_{r}}=r$, and so $\{f^{-1}(x)\,:\, x\in[0,1]\}$
equals $\{C_{r}\}$ up to measure $0$. 

Fix a sequence $(\varepsilon_{k})$ to be determined later. For $i=0,1$
let $C_{i}$ be disjoint closed sets with null boundary and $\mu(C_{i}\Delta P_{i})<\varepsilon$.
Set $U_{0}=\emptyset$ and $U_{1}=[0,1]\setminus C_{1}$ 

Let $\mathbb{D}_{k}\subseteq\mathbb{D}$ be the set of reduced dyadic
rationals with denominator $2^{k}$. We proceed by induction on $k$,
defining at each step the sets $U_{r},C_{r}$ for $r\in\mathbb{D}_{k}$
under the assumption that they have been defined already for $r\in\cup_{j<k}\mathbb{D}_{j}$.
Write $\mathbb{E}_{k}=\cup_{j<k}\mathbb{D}_{j}=\{r_{1},\ldots,r_{n}\}$
with $r_{1}<\ldots<r_{n}$ and let $r\in\mathbb{D}_{k}$. Then there
are $r',r''\in\mathbb{E}_{k}$ with $r'<r<r''$ and $(r',r'')\cap\mathbb{E}_{k}=\emptyset$.
Let $V=U_{r''}\setminus\overline{U}_{r'}$ and choose $C_{r}\subseteq V$
with $\mu(C_{r})>(1-\varepsilon_{k})\mu(V)=(1-\varepsilon_{k})\mu(U_{r''}\setminus U_{r'})$.
Choose $U_{r}$ such that it contains $C_{r}\cup U_{r'}$, it has
$\mu(\partial U_{r})=0$ and $\overline{U}_{r}\subseteq U_{r''}$.

Write $\mathcal{Q}=\{C_{r}\}_{r\in\mathbb{D}}$. Set $\widetilde{C}_{k}=\cup_{i\geq k}\cup_{r\in\mathbb{D}_{i}}C_{r}$
and let $\mathcal{Q}_{k}=\{C_{r}\}_{r\in\mathbb{E}_{k}}\cup\{\widetilde{C_{k}}\}$
be the partition obtained by merging all the atoms $C_{r}$ in $\mathcal{Q}$
with $r\in\cup_{j\geq k}D_{j}$. Let $C_{k}^{*}=\cup_{r\in D_{k}}C_{r}$.
The sequence $(\varepsilon_{k})$ controls the convergence of the
sequence $(\mu(C_{k}^{*}))$ to $1$, and the latter can be made to
converge arbitrarily quickly. In particular we can guarantee that
$Q$ has finite entropy. Now $\mathcal{Q}_{k}\rightarrow\mathcal{Q}$
in the Rohlin metric, so \begin{eqnarray*}
d(\mathcal{P},\mathcal{Q}) & = & \lim_{k\rightarrow\infty}d(\mathcal{P},\mathcal{Q}_{k})\\
 & \leq & \lim_{k\rightarrow\infty}(d(\mathcal{P},\mathcal{Q}_{1})+\sum_{i=1}^{k-1}d(\mathcal{Q}_{i},\mathcal{Q}_{i+1}))\\
 & = & d(\mathcal{P},\mathcal{Q}_{1})+\sum_{i=1}^{\infty}d(\mathcal{Q}_{i},\mathcal{Q}_{i+1})\end{eqnarray*}
and the last line can be made arbitrarily small by prudent choice
of $(\varepsilon_{k})$, since $\mathcal{Q}_{i+1}$ refines $\mathcal{Q}_{i}$
by splitting $C_{k}^{*}$ into at most $2^{k}$ atoms whose relative
mass is determined by $\varepsilon_{k}$.
\end{proof}
We can now prove theorem \ref{thm:Zd-determinism}. Note that even
for $d=1$ this proof is more direct than that given in \cite{KaminskiSiemaszkoSzymanski03}.
\begin{proof}
(of theorem \ref{thm:Zd-determinism}). Let $\mathbb{Z}^{d}$ act
on $X$ and suppose that for every $f\in C(X)$ one has \[
f\in\left\langle 1,T^{u}f\,:\, u<0\right\rangle \]
where $<$ is the lexicographical order on $\mathbb{Z}^{d}$. This
implies that $f$ is measurable with respect to the $\sigma$-algebra
generated by $\{T^{u}f\,:\, u<0\}$, and in particular this shows
that for any $T$-invariant measure $\mu$ there is a dense (in the
Rohlin metric) set of partitions $\mathcal{Q}$ for which $h(\mathcal{Q})=0$,
namely those which come from continuous functions (proposition \ref{pro:continuous-partitions-are-dense}).
Since $h(\mu,\mathcal{P})$ is continuous in $\mathcal{P}$ under
the Rohlin metric we conclude that $h(\mu,\mathcal{P})=0$ for every
two-set partition and hence $h(\mu)=0$. By the variational principle,
$h_{\textrm{top}}(T)=0$.
\end{proof}

\section{\label{sec:Prediction-in-symbolic-systems}Prediction in symbolic
systems}

\subsection{Generalities about subshifts and prediction}

Let $\Sigma$ be a finite alphabet, $\sigma:\Sigma^{\mathbb{Z}}\rightarrow\Sigma^{\mathbb{Z}}$
the shift transformation. For $x\in\Sigma^{\mathbb{Z}}$ set $x^{-}=(\ldots,x_{-2},x_{-1})$,
and for a subshift $X\subseteq\Sigma^{\mathbb{Z}}$ let $X^{-}=\{x^{-}\,:\, x\in X\}$.
A finite or right-infinite word $a$ is an extension of $x^{-}\in X^{-}$
if $x^{-}a$ appears in $X$. Let $L(X)$ be the set of finite words
appearing in $X$ and $L_{m}(X)=L(X)\cap\Sigma^{m}$. 

The following fact is well-known:
\begin{lem}
\label{lem:no-perfect-prediction}A subshift $X$ is the union of
periodic orbits if and only if every $x^{-}\in X^{-}$ extends uniquely
to $x\in X$.\end{lem}
\begin{proof}
If $X$ is a finite union of periodic orbits the conclusion is clear. 

For the converse, we rely on the simple fact that, if there is some
$n$ such that $x_{-n},\ldots,x_{-1}$ determines $x_{0}$ for all
$x\in X$, then $X$ is the finite union of periodic orbits. Thus
if $X\subseteq\Sigma^{\mathbb{Z}}$ is not the union of periodic orbits,
then for every $n$ there is a word $a_{n}\in L_{n}(X)$ and distinct
symbols $u_{n},v_{n}\in\Sigma$ such that $a_{n}u_{n},a_{n}v_{n}\in L_{n+1}(X)$.
Therefore there are words $b_{n},c_{n}\in\Sigma^{\mathbb{N}^{+}}$
beginning with $u_{n},v_{n}$ respectively such that $a_{n}b_{n},a_{n}c_{n}$
appear in $X$. By compactness, we can choose a subsequence $n(k)$
such that $u=u_{n(k)}$ and $v=v_{n(k)}$ are constant, $a_{n(k)}\rightarrow x^{-}\in X^{-}$,
$b_{n(k)}\rightarrow b\in\Sigma^{\mathbb{N}^{+}}$ and $c_{n(k)}\rightarrow c\in\Sigma^{\mathbb{N}^{+}}$.
But then $a,b$ begin with the distinct symbols $u,v$ and $x^{-}a,x^{-}b\in X$,
so $x^{-}$ has at least two extensions in $X$.
\end{proof}
Thus, every infinite subshift, including zero-entropy ones, has at
least one past with multiple extensions. On the other hand, the following
observation was pointed out to us by B. Weiss. Note that it is is
a special case of the general fact that minimal systems are invertible
on a dense $G_{\delta}$. 
\begin{lem}
\label{lem:prediction-in-minimal-shifts}If $X$ is a minimal subshift
then for every $a\in L(X)$ and $k\in\mathbb{N}$ there is a word
$b\in L(X)$ such that $ba\in L(X)$, and every occurrence of $ba$
in $X$ is followed by a unique word $c\in\Sigma^{k}$.\end{lem}
\begin{proof}
It suffices to show this for $k=1$, as the general case then follows
by induction. Let $a\in L(X)$ and $u\in\Sigma$ such that $au\in L(X)$.
Consider all $b$'s such that $bu\in L(X)$ and $au$ appears in $bu$
exactly twice, as a front segment and a back segment. By minimality
the lengths of such $b$'s is bounded above and we can choose a maximal
such $b$. If $x^{+}\in X^{+}$ and $bx^{+}\in X^{+}$, then by minimality
$au$ appears in $x^{+}$; thus by maximality of $b$ we must have
$x^{+}(1)=u$, for otherwise there is a front segment $c$ of $x^{+}$
such that $au$ appears in $bc$ only as a front and back segment,
which is impossible by maximality of $b$. Thus $b$ is always followed
by $u$ in $X$.\end{proof}
\begin{cor}
\label{cor:prediction-word}If $X$ is a minimal subshift and $u\in L(X)$
then there is a word $v\in L(X)$ such that every occurrence of $v$
is followed by $u$.\end{cor}
\begin{proof}
Let $u$ be given, let $k$ be large enough that every $c\in L_{k}(X)$
contains $u$. In the previous lemma let $a$ be the empty word, and
let $b,c$ be the words obtained. Then $b$ is always follows by $c$
and $c=c'uc''$ for some $c',c''$. The word $v=bc'$ has the desired
property.
\end{proof}

\subsection{Realization theorem}

We now begin the proof of theorem \ref{thm:symbolic-realization}.
We start with a measure preserving system $(X,\mathcal{B},\mu,T)$
of entropy zero, and wish to construct a strictly ergodic subshift,
supporting an isomorphic measure, in which each past has at most two
futures. We may assume $\mu$ is aperiodic (i.e. the set of periodic
points has measure 0); otherwise the statement is trivial. By e.g.
\cite{Weiss98}, we may assume that $\mu$ is an invariant measure
on a uniquely ergodic, topologically weak mixing, minimal subshift
$X\subseteq\{0,1\}^{\mathbb{Z}}$.

We construct a sequence of two-set generating clopen partitions $\mathcal{P}_{n}$
for $n=0,1,2,\ldots$ such that $\mathcal{P}_{n}\rightarrow\mathcal{P}_{*}$,
where $\mathcal{P}_{*}$ generates for $\mu$. Denote by $X_{n}$
the symbolic system arising from $X$ and $\mathcal{P}_{n}$. Note
that since $\mathcal{P}_{n}$ is clopen, $X_{n}$ is minimal and uniquely
ergodic. The two-sided $\mathcal{P}_{n}$-name of a point $x\in X$
is a point in $X_{n}$.

We will define a sequence of integers $m(n)\geq n$ such that $L_{m(n)}(X_{n})=L_{m(n)}(X_{n+1})$,
and another sequence $k(n)\geq n$ with the property that for every
$u\in\Sigma^{k(n)}$,\[
\#\{w\in\Sigma^{n}\,:\, uw\in L(X_{n})\}\leq2\]
these numbers will satisfy $m(n)\geq k(n)+n$, so that the system
$X_{*}$ arising from $\mathcal{P}_{*}$ will have the property that
for every $u\in\Sigma^{k(n)}$,\[
\#\{w\in\Sigma^{n}\,:\, uw\in L(X_{*})\}\leq2\]
This implies the desired result. By choosing the $m(n)$ large enough
at each stage, we can furthermore guarantee that $X_{*}$ is minimal
and uniquely ergodic, but for simplicity we do not go into the details
of this.

The construction is by induction. Define $\mathcal{P}^{(0)}$ to be
the clopen generating partition according to the $0$-th symbol, set
$m(0)=0$ and $k(0)=0$.

We describe now the inductive step of the construction. We are given
a two-set generating partition $\mathcal{P}_{n}$ of $X$ into clopen
sets and an integer $m(n)$. Given $\varepsilon>0$ we will construct
a new partition $\mathcal{P}_{n+1}$ which is $\varepsilon$-close
to $\mathcal{P}_{n}$. We will ensure that $L_{m(n)}(X_{n})=L_{m(n)}(X_{n+1})$
and define an integer $k(n+1)$ with the properties above. Finally
we will be free to choose $m(n+1)$ arbitrarily, since it only affects
the next step of the construction.

Let $Y_{n}$ be the shift of finite type whose allowed blocks of length
$m(n)+1$ are those appearing in $L_{m(n)+1}(X_{n})$. Since $X_{n}$
is infinite and transitive, and $X_{n}\subseteq Y_{n}$, it follows
from basic properties of shifts of finite type that $Y_{n}$ has positive
entropy. Using the fact that $X_{n}$ is mixing and has zero entropy
(whereas $Y_{n}$ has positive entropy) we can find a word $a\in L_{m(n)+1}(X_{n})$,
a word $b_{\textrm{old}}\in L(X_{n})$ and a word $b_{\textrm{new}}\in L(Y_{n})\setminus L(X_{n})$
such that $b_{\textrm{old}},b_{\textrm{new}}$ have the same length,
and both begin and end with the word $a$. Furthermore, using standard
marker arguments (see e.g. \cite{LindMarcus1995}), we may assume
that if $x\in X_{n}$ and we replace some sequence of occurrence of
$b_{old}$ in $x$ with $b_{new}$, and if these occurrences were
at least $2\ell(b_{old})$ apart, then we can identify the location
of the changes from the modified sequence.

The partition of $\mathcal{P}_{n+1}$ will be constructed by replacing
some of the occurrences of $b_{\textrm{old}}$ in $X_{n}$ with $b_{\textrm{new}}$.
This is done as follows. First, using Corollary \ref{cor:prediction-word},
choose $c\in L(X_{n})$ such that every time $c$ appears in $X_{n}$
it is followed by $b_{\mbox{old}}$. We can extend $c$ backwards
arbitrarily while preserving this property, so we may assume that
$c$ is arbitrarily long. Since $X_{n}$ is minimal, there is an $R$
such that the gap between occurrences of $c$ in $X_{n}$ is at most
$R$. 

Next, choose a large $N$ (how large will depend on $R,\ell(b_{\mbox{old}})$
and on the growth of words in the system $X_{n}$, and will be explained
below) and choose a clopen bounded Alpern tower in $X_{n}$ all of
whose columns are of height $N-1$ or $N$, and such that the base
is contained in the cylinder set defined by $cb_{old}$. Purify each
column of the tower according to the clopen partition $\vee_{i=0}^{4N}T^{-i}\mathcal{P}_{n}$.
Consider one such column, which corresponds to the $\mathcal{P}_{n}$-name
$w$. We proceed to modify the $\mathcal{P}_{n}$-name of the column;
doing this for each column defines a new partition $\mathcal{P}_{n+1}$.

Fix $x\in X$ and its corresponding column. Let $i(1)=0$ denote the
height in the column of the first occurrence of $cb_{\textrm{old}}$
in $w$, let $i(2)$ be the index of the next occurrence which does
not intersect the first occurrence, and so on until $i(r)$, the index
of the last occurrence of $cb_{old}$ which is contained completely
in the current column. Replace the occurrences of $cb_{\textrm{old}}$
at indices $i(1),i(2)$ with $cb_{\textrm{new}}$. 

Using the syndeticity of occurrences of $c$, for some $\alpha>0$
we have $r\geq\alpha N$, where $\alpha$ depends on $R$ but not
$N$. We next encode the $\mathcal{P}_{n}$-name of $x$ from time
$0$ to $4N$. We do so by replacing the word $cb_{\textrm{old}}$
at some of the levels $i(4),i(6),\ldots,i(r-2)$ with $cb_{\textrm{new}}$.
We use only locations $i(j)$ where $j$ is even; thus no new consecutive
occurrences of $cb_{new}$ are introduced, and the consecutive occurrences
of $cb_{\textrm{new}}$ at the bottom of the column are unique and
serve to identify it. We can encode the atom of $\vee_{i=0}^{4N}T^{-i}\mathcal{P}_{n}$
to which $x$ belongs in the approximately $\frac{1}{2}\alpha N$
bits available because $h(X_{n})=0$, so the number of $\vee_{i=0}^{4N}T^{-i}\mathcal{P}_{n}$-names
is $<2^{\alpha N/4}$ assuming $N$ is large enough.

We have defined a partitions $\mathcal{P}_{n+1}$. Note that we have
modified $w$ along a set of density at most $\ell(b_{\textrm{old}})/\ell(cb_{\textrm{old}})$,
which can be made arbitrarily small by making $c$ long; thus $\mathcal{P}_{n+1}$
can be made $\varepsilon$-close to $\mathcal{P}_{n}$. 

Since $b_{\textrm{new}}$ does not appear in $L(X_{n})$, we can recover
the $\mathcal{P}_{n}$ name of a point $x\in X$ simply by replacing
every occurrence of $cb_{\textrm{new}}$ with $cb_{\textrm{old}}$.
Thus, since $\mathcal{P}_{n}$ generates, so does $\mathcal{P}_{n+1}$.

Because $b_{\textrm{old}},b_{\textrm{new}}$ agree on their first
and last $m(n)$ symbols, and because $b_{\textrm{new}}\in Y_{n}$
and all $m(n)$-blocks in $Y_{n}$ are in $L_{n}(X_{n})$, we also
have $L_{m(n)}(X_{m})\subseteq L_{m(n)}(X_{n+1})$.

Consider a point $x\in X$ . We will show that by looking $2N$ symbols
into the past of the $\mathcal{P}_{n+1}$-name of $x$, we can determine
that the $\mathcal{P}_{n+1}$-name of $x$ from time $1$ to $\ell(b_{\textrm{new}})$
takes on one of at most two possible values. Thus setting $k(n)=2N$
and noting that $\ell(b_{\textrm{new}})\geq m(n)\geq n$ we will have
completed the inductive step. 

Look into the $\mathcal{P}_{n+1}$-past of $x$ until we find a sequence
of two consecutive occurrences of $cb_{\textrm{new}}$; this must
happen after at most $N$ symbols at some index $i$. Looking back
at most $N$ symbols more we find the next group of two or five consecutive
$cb_{\textrm{new}}$'s at some index $j$. Between $j$ and $i$ we
have coded the $\mathcal{P}_{n}$ name of $x$ from times $j$ to
time $j+3N$ (and even a little bit more). In any case, assuming as
we may that $N>\ell(b_{\textrm{new}})$, and since $j\geq-2N$, we
can certainly recover the $\mathcal{P}_{n}$ name of $x$ from time
$j$ to time $\ell(b_{\textrm{new}})$.

We now claim that there are at most two choices for the $\mathcal{P}_{n+1}$-name
of $x$ from time $1$ to $m(n+1)$. Note that the $\mathcal{P}_{n}$-name
of $x$ and the $\mathcal{P}_{n+1}$-name of $x$ differ only at points
which lie in the $\ell(b_{\textrm{new}})$ symbols following certain
occurrences of $c$. But if some such occurrence of $c$ intersects
the $\mathcal{P}_{n}$-name of $x$ from times $-\ell(b_{\textrm{new}})+1$
to $\ell(b_{\textrm{new}})$, then from space considerations there
is a unique such $c$; and in this case the next $\ell(b_{\textrm{new}})$
symbols of $x$ are either $b_{\textrm{new}}$ or $b_{\textrm{old}}$.
Thus there are at most two possible choices for the atom of $\vee_{s=1}^{m(n)+1}T^{s}\mathcal{P}_{n+1}$
to which $x$ belongs. 

This completes the discussion of the induction step. By choosing $\varepsilon$
small enough at each stage we can arrange that $\mathcal{P}_{n}\rightarrow\mathcal{P}_{*}$
with $\mathcal{P}_{*}$ a generating partition for $\mu$, and $X_{*}$
will be $2$-branching. By a proper choice of $m(n)$ and using the
unique ergodicity and minimality of $X$ (and hence of all the $X_{n}$),
we can also ensure that $X_{*}$ is minimal and uniquely ergodic.

\section{\label{sec:extremely-non-invertible-examples}An extremely non-invertible
zero-entropy system}

\subsection{Generalities}

In this section we address the relation between entropy and the structure
of preimage sets of points in non-invertible topological systems.
The motivation for this is the following simple fact, whose proof
is a good illustration of why one expects there to be a connection
between entropy and large preimage sets:
\begin{prop}
\label{pro:large-preimages-implies-entropy}A system with no small
preimages has entropy at least $\log2$.\end{prop}
\begin{proof}
Let $(X,T)$ be a system and $\delta>0$ such that for every $x\in X$
there are $x',x''\in T^{-1}(x)$ with $d(x',x'')>\delta$. We can
define functions $\tau_{0},\tau_{1}:X\rightarrow X$ such that $\tau_{0}(x),\tau_{1}(x)\in T^{-1}(x)$
and $d(\tau_{0}(x),\tau_{1}(x))>\delta$; note that $\tau_{0},\tau_{1}$
need 7 not be continuous. For $n\in\mathbb{N}$ and a sequence $a=a_{n}a_{n-1}\ldots a_{1}\in\{0,1\}^{n}$
let \[
\tau_{a}(x)=\tau_{a_{n}}(\tau_{a_{n-1}}(\ldots\tau_{a_{1}}(x)\ldots))\]
Note that $T(\tau_{a}(x))=\tau_{b}(x)$ where $b\in\{0,1\}^{n-1}$
is obtained by deleting the first symbol of $a$. 

For a fixed $x\in X$ consider the set \[
A_{n}(x)=\{\tau_{a}(x)\,:\, a\in\{0,1\}^{n}\}\]
If $a,b\in\{0,1\}^{n}$ and $a\neq b$ then there is a maximal index
$i<n$ such that $a_{j}=b_{j}$ for $1\leq j\leq i$ but $a_{i+1}\neq b_{i+1}$.
Let $y=\tau_{a_{i}a_{i-1}\ldots a_{1}}(x)=\tau_{b_{i}b_{i-1}\ldots b_{1}}(x)$;
then\begin{eqnarray*}
T^{n-i-1}(\tau_{a}(x)) & = & \tau_{a_{i+1}}(y)\\
T^{n-i-1}(\tau_{b}(x)) & = & \tau_{b_{i+1}}(y)\end{eqnarray*}
so $d(T^{n-i+1}\tau_{a}(x)),T^{n-i+1}\tau_{b}(x))>\delta$. It follows
that all the points in $A_{n}(x)$ are distinct and the set $A_{n}(x)$
is $(n,\delta)$-separated; since this is true for all $n$, this
implies that $h(X,T)>\log2$.
\end{proof}
One easy consequence of this is that for finite alphabets $\Sigma$
every extremely non-invertible subshift of $\Sigma^{\mathbb{Z}}$
has entropy at least $2$, because once a metric is fixed there is
a $\delta$ such that every two distinct preimages of a point are
$\delta$ apart. 

As was mentioned in the introduction, J. Bobok has shown that for
maps of the interval if a map is $k$-to-one then it has entropy $>\log k$
\cite{Bobok02}. 

It is not hard to construct examples of zero entropy systems where
every point has multiple preimages, but it is not so easy to construct
such a system with a globally supported ergodic measure, and Eli Glasner
has asked whether this is possible. The construction below gives an
affirmative answer to this question.

\subsection{\label{sub:non-inv--construction}The construction}

Let $\sigma$ be the shift on the one-sided Bebutov system $[0,1]^{\mathbb{N}}$.
We will construct a subshift of the Bebutov system by specifying a
point $x_{*}\in[0,1]^{\mathbb{N}}$ and taking its orbit closure $X=\overline{\{\sigma^{n}x_{*}\}_{n\in\mathbb{N}}}$.
Things will be engineered so that $X$ has zero topological entropy,
and $x_{*}$ is generic for an ergodic measure $\mu$ on $X$ having
support $X$.

For words $x,y\in[0,1]^{\mathbb{N}}$ we set \[
d(x,y)=\sum_{i=1}^{\infty}|x(i)-y(i)|\cdot2^{-i}\]
this defines a metric on $[0,1]^{\mathbb{N}}$ which is compatible
with the compact product topology. We also write \[
\left\Vert x\right\Vert =d(x,\overline{0})\]
where $\overline{0}=(0,0,\ldots)$. For a finite word $x$ we define
\[
\left\Vert x\right\Vert =\sum_{i=1}^{\ell(x)}|x(i)|\cdot2^{-i}=\inf\left\{ \left\Vert y\right\Vert \,:\, y\in[0,1]^{\mathbb{N}}\textrm{ and }x\textrm{ is a front segment of }y\right\} \]
 Note that $\left\Vert ab\right\Vert \geq\left\Vert a\right\Vert $
and that if $x_{n}$ are finite words and $x_{n}\rightarrow x\in[0,1]^{\mathbb{N}}$
in the obvious sense then $\left\Vert x_{n}\right\Vert \rightarrow\left\Vert x\right\Vert $.

Suppose $x\in[0,1]^{*}$ is a finite word. We define $\theta_{0}(x),\theta_{1}(x)\in[0,1]$
by\[
\theta_{0}(x)=\frac{1}{8}\left\Vert x\right\Vert \qquad,\qquad\theta_{1}(x)=\frac{1}{4}\left\Vert x\right\Vert \]
and we define $\tau_{0},\tau_{1}:[0,1]^{*}\rightarrow[0,1]^{*}$ by\[
\tau_{0}(x)=\theta_{0}(x)x\qquad,\qquad\tau_{1}(x)=\theta_{1}(x)x\]
i.e. the symbols $\theta_{i}(x)$ are appended to the beginning of
$x$. 

For a sequence $b=b_{M}b_{M-1}\ldots b_{1}\in\left\{ 0,1\right\} ^{M}$
define $\tau_{b}$ inductively by\[
\tau_{b_{M}\ldots b_{1}}(x)=\tau_{b_{M}}(\tau_{b_{M-1}\ldots b_{1}}(x))\]
and set $T_{\emptyset}(x)=x$. Note that if $b=b_{M}\ldots b_{1}$
then \[
\sigma^{i}(\tau_{b}(x))=\tau_{b_{M-i}\ldots b_{1}}(x)\]
and in particular $\sigma^{M}(\tau_{b}(x))=x$. One verifies that
$\left\Vert \tau_{b}(x)\right\Vert \rightarrow0$ exponentially as
the length of $b$ tends to $\infty$, uniformly in $b$ and $x$. 

We define $\tau_{b}$ on $[0,1]^{\mathbb{N}}$ by the same formula.
In the subshift we are about to construct the preimage set of a point
$x$ will contain at least $\tau_{0}(x),\tau_{1}(x)$. Since $\tau_{b}(x)\rightarrow\overline{0}$
as $\ell(b)\rightarrow\infty$ the preimage tree of each point will
be {}``narrow'', and not contribute to the entropy. Note however
that there will also be preimages which do not come from applications
of $\tau_{b}$.

We construct $x_{*}$ in recursively. At the $n$-th stage we will
be given a finite word $x_{n}$ of length $L_{n}$ and construct a
word $x_{n+1}$ of length $L_{n+1}$ such that $x_{n+1}=x_{n}x'_{n}$
for some word $x_{n}'$. We then take $x_{*}$ to be the limit of
this increasing sequence of finite words.

We begin with an arbitrary finite word $x_{0}$ of length $L_{0}>0$.
Our only assumption about $x_{0}$ is that it is strictly positive.

The passage from stage $n$ to $n+1$ is as follows. Given $x_{n}$
of length $L_{n}$, for $0\leq k<L_{n}$ let $w_{k}$ be the back
segment of $x_{n}$ starting at index $k$, that is,\[
w_{k}=x_{n}(k)x_{n}(k+1)\ldots x_{n}(L_{n})\]
so $\ell(w_{k})=L_{n}-k+1$. For $b\in\left\{ 0,1\right\} ^{3^{L_{n}}}$
set\[
w_{b,k}=\tau_{b}(w_{k})\]
Define $y_{n}$ to be some concatenation of the words $w_{b,k}$ as
$b$ varies over $\left\{ 0,1\right\} ^{3^{L_{n}}}$ and $0\leq k<L_{n}$
(the order is not important).

Now choose a large integer $M_{n}$ which we will specify later. For
now we note that $M_{n}$ may be chosen to depend not only on all
the previous stages but also on $y_{n}$. Define\[
\begin{array}{ccc}
x_{n+1} & = & \underbrace{(x_{n}x_{n}\ldots x_{n})}y_{n}\\
 &  & M_{n}\textrm{ times}\quad\end{array}\]

Set $x_{*}=\lim x_{n}$ and let $X$ be the orbit closure of $x_{*}$.
In the next few subsections we will show that $(X,\sigma)$ has the
advertised properties.

\subsection{$(X,\sigma)$ is extremely non-invertible}

The point $x_{*}$ has been constructed in such a way that if some
finite word $a$ appears in $x_{*}$ then it appears in at least two
different configurations, preceded by symbols $r,r'\in[0,1]$ such
that $|r-r'|\geq\frac{1}{16}\left\Vert a\right\Vert $. This is because
if $a$ is a subword of $x_{n}$ then $a$ is a front segment of some
back segment $b$ of $x_{n}$, and so $\tau_{0}(b)$ and $\tau_{1}(b)$
appear in $x_{n+1}$, and by definition the first symbol of $\tau_{0}(b)$
and $\tau_{1}(b)$ differ by $\frac{1}{16}\left\Vert b\right\Vert $,
and $\left\Vert b\right\Vert \geq\left\Vert a\right\Vert $.

Thus if $y$ is a limit point of $x_{*}$ and $y\neq0$, then $y$
is a limit point of finite subwords $a_{n}$ of $x_{*}$, and since
$\left\Vert y\right\Vert >c>0$ for some $c$ we have that $\left\Vert a_{n}\right\Vert >c$
for all large enough $n$. Therefore we can find symbols $r'_{n},r''_{n}\in[0,1]$
such that $|r'_{n}-r''_{n}|>\frac{1}{16}c$ and $r'_{n}a_{n},r''_{n}a_{n}$
appear in $x_{*}$. Passing to a subsequence we get that $r'_{n}a_{n}\rightarrow r'y$
and $r''_{n}a_{n}\rightarrow r''y$ for some $r',r''\in[0,1]$ with
$|r'-r''|\geq\frac{1}{16}c$, and so $r'y,r''y$ are distinct preimages
of $y$ in $X$.

It remains to check that $\overline{0}$ has two preimages (it is
clear from the construction that $\overline{0}\in X$, since $x_{*}$
has arbitrarily long sequences of small numbers, consisting of front
segments of the $w_{b,k}$). Since $\overline{0}$ is a fixed point
of $\sigma$, one preimage is $\overline{0}$ itself. To see that
there are other preimages, note that the words $x_{n}$ all end in
the same positive letter $\varepsilon$, the last letter of $x_{0}$,
and this is also the last letter of all the words $w_{b,k}$ we constructed
at each stage. On the other hand as $\ell(b)\rightarrow\infty$ the
front segments of $w_{b,k}$ approach $\overline{0}$, so there are
arbitrarily long sequences of arbitrarily small numbers in $x_{*}$,
each sequence preceded by an occurrence of $\varepsilon$. Thus $\varepsilon000\ldots$
is also a preimage of $0$ in $X$.

\subsection{$(X,\sigma)$ has zero topological entropy}

We verify this by estimating the number of $\varepsilon$-separated
orbits. For words $a,a'$ (either finite or infinite) we write \[
\left\Vert a-a'\right\Vert _{\infty}=\sup_{i}|a(i)-a'(i)|\]
Note that for $x,x'\in X$, \[
\left\Vert x|_{[1;n]}-x'|_{[1;n]}\right\Vert _{\infty}>\varepsilon\;\;\implies\;\;\max\{d(T^{i}x,T^{i}x')\,:\, i=1,\ldots,n\}>\varepsilon\]

Fix $\varepsilon>0$, and let $A_{n}$ be the set of all subwords
of $x_{*}$ of length $n$. Set \[
C_{\varepsilon}(n)=\max\left\{ |A|\,:\, A\subseteq A_{n}\,,\,\forall a,a'\in A\,\left\Vert a-a'\right\Vert _{\infty}>\varepsilon\right\} \]
The topological entropy of $(X,S)$ is \[
\lim_{\varepsilon\rightarrow0}\,\limsup_{n\rightarrow\infty}\frac{1}{n}\log C_{\varepsilon}(n)\]

For a finite or infinite word $a$ with symbols in $[0,1]$, let $[a]_{\varepsilon}$
denote the word $b$ of the same length such that \[
b(i)=[a(i)/\varepsilon]\cdot\varepsilon\]
(here $[r]$ denoted the integer part of $r$). Thus the coordinates
of $[a]_{\varepsilon}$ belong to the finite set $\{0,\varepsilon,2\varepsilon,\ldots,[\frac{1}{\varepsilon}]\varepsilon\}$.
Note that if $\left\Vert a-a'\right\Vert _{\infty}\geq\varepsilon$
then $\left\Vert [a]_{\varepsilon/2}-[a']_{\varepsilon/2}\right\Vert _{\infty}\geq\varepsilon/2$.
It is therefore sufficient to prove the following:
\begin{claim}
\label{claim:word-growth-rate}For every $\varepsilon>0$, the number
of length $n$ subwords of $[x_{*}]_{\varepsilon/2}$ which are at
least $\varepsilon/2$ apart in $\left\Vert \cdot\right\Vert _{\infty}$
grows sub-exponentially with $n$.
\end{claim}
We will use the following property of $x_{*}$:
\begin{lem}
\label{lemma:structure-of-x-star}For every $n$ we can write $x_{*}=a_{1}a_{2}a_{3}\ldots$
, where each $a_{i}$ is of length at least $3^{L_{n}}$ and for each
$i$, either
\begin{enumerate}
\item $a_{i}=x_{n}$, or
\item For each $1\leq j\leq\ell(a_{i})-L_{n}$ we have $a_{i}(j)\leq\frac{7}{8}a_{i}(j+1)$.
\end{enumerate}
In particular, for any $\varepsilon>0$, for $n$ large enough each
$a_{i}$ is either equal to $x_{n}$ or else all the coordinates of
$a_{i}$, except the last $2L_{n}$ coordinates, are of magnitude
$<\varepsilon$.
\end{lem}
The proof of the lemma is an elementary induction from the definitions,
and is omitted.
\begin{proof}
(of claim \ref{claim:word-growth-rate}) Fix $\varepsilon>0$ and
let $z_{*}=[x_{*}]_{\varepsilon/2}$ and $z_{m}=[x_{m}]_{\varepsilon/2}$.
From the lemma , we see that for the given $\varepsilon$ for large
enough $m$ we can write \[
z_{*}=v_{1}v_{2}v_{3}\ldots\]
 and for each $i$ the word $v_{i}$ is either equal to $z_{m}$,
or else $\ell(v_{i})\geq3^{L_{m}}$ and at least a $(1-2^{-L_{m}})$-fraction
of the coordinates of $v_{i}$ are $0$. In view of this, the fact
that the number of subwords of $z_{*}$ of length $n$ grows sub-exponentially
is now a standard counting argument, and the claim follows. This shows
that $h_{\textrm{top}}(X,\sigma)=0$.
\end{proof}

\subsection{$x_{*}$ is generic for a globally-supported measure $\mu$ on $X$.}

A point $y$ in a dynamical system $(Y,S)$ is a generic point for
a measure $\mu$ if for every continuous function $f\in C(Y)$ it
holds that $\lim_{N\rightarrow\infty}\frac{1}{N}\sum_{n=1}^{N}f(S^{i}y)$
exists. When this is true then $\frac{1}{N}\sum_{n=1}^{N}\delta_{S^{i}y}$
converges in the weak-$*$ topology to an invariant measure $\mu$
on $Y$ (here $\delta_{x}$ is the point mass at $x$). One condition
that guarantees that $y$ is generic is that for every open set $U\subseteq Y$
the averages $\lim_{N\rightarrow\infty}\frac{1}{N}\sum_{n=1}^{N}1_{U}(S^{i}y)$
exist; in fact it is sufficient to verify this for $U$ coming from
a basis for the topology of $X$.

For $U\subseteq[0,1]^{k}$, let \[
[U]=U\times[0,1]^{\mathbb{N}\setminus\{1,\ldots,k\}}\subseteq[0,1]^{\mathbb{N}}\]
be the cylinder determined by $U$. Sets of this form for open $U$
constitute a basis for the topology of $[0,1]^{\mathbb{N}}$. We will
show that for every such $U$, the sequence \begin{equation}
p(m)=\frac{1}{m}\sum_{i=1}^{m}1_{[U]}(\sigma^{i}x_{*})\label{eq:U-density-limit}\end{equation}
converges. This implies that the weak$^{*}$ limit measure \[
\mu=\lim_{n\rightarrow\infty}\frac{1}{n}\sum_{i=1}^{n}\delta_{\sigma^{i}x_{*}}\]
exists, and is a shift-invariant measure on $X$. In fact, we will
show that $\mu(U)>0$ if and only if $p(n)>0$ for some $n$. From
this it will follow that $\mu$ has global support in $X$.

For a finite word $a$ we will say that $a\in[U]$ if $ab\in[U]$
for every infinite $b\in[0,1]^{\mathbb{N}}$. Thus if $a\in[U]$ then
$ab\in[U]$ for every finite $b$. The property $a\in[U]$ depends
only on the first $k$ coordinates of $a$ (recall that $U\subseteq[0,1]^{k}$).
Note that if $\ell(a)<k$ it is possible that $a\notin[U]$ but that
$ab\in[U]$ for some (finite of infinite) $b$.
\begin{claim}
Let $U\subseteq[0,1]^{k}$ and $p(n)$ as above. The limit $\lim_{s\rightarrow\infty}p(L_{s})$
exists; furthermore, if $p(n)>0$ for some $n$ then the limit is
positive.\end{claim}
\begin{proof}
If $\sigma^{n}x_{*}\notin[U]$ for every $n$ then clearly $\lim p(n)=0$.
Therefore we must check only the case when $\sigma^{n}x_{*}\in[U]$
for some $n$. Note that in this case, $p(m)>0$ for all $m\geq n$.
We prove first that $p(L_{r})$ converges at $r\rightarrow\infty$,
and then the general claim.

For a word $a$, let $I(a)$ be the number of indices $0\leq n<\ell(a)$
such that $\sigma^{n}a\in[U]$. If we let $a_{m}$ be the front $m$-segment
of $x_{*}$, we have\[
\frac{I(a_{m})}{m}\leq p(m)\leq\frac{I(a_{m})+k}{m}\]
(the right inequality is because of edge effects; it is possible for
$\sigma^{n}a\notin[U]$ but $\sigma^{n}x_{*}\in[U]$ if $\ell(a)-k<n<\ell(a)$).
In particular, for any $r$ we have\begin{equation}
\frac{I(x_{r})}{L_{r}}\leq p(L_{r})\leq\frac{I(x_{r})+k}{L_{r}}\label{eq:density-estimate}\end{equation}

If $p(L_{r})>0$ then also $p(L_{r+1})>0$, and $x_{r+1}$ contains
at least $M_{r}$ copies of $x_{r}$. Thus if we assume that $M_{s}\geq2^{s}$
for every $s$, we may fix $r$ such that $I(x_{s})\geq2^{s}$ for
every $s\geq r$.

For an $s$ as above, write \[
x_{s+1}=x_{s}x_{s}\ldots x_{s}y_{s}\]
 as in the construction of $x_{s+1}$, with the $x_{s}$'s repeating
$M_{s}$ times. We can write $I(x_{s+1})=I_{1}+I_{2}$, where \begin{eqnarray*}
I_{1} & = & \#\left\{ 0\leq n<M_{s}L_{s}\,:\,\sigma^{n}x_{s+1}\in[U]\right\} \\
I_{2} & = & \#\left\{ M_{s}L_{s}\leq n<L_{s+1}\,:\,\sigma^{n}x_{s+1}\in[U]\right\} \end{eqnarray*}
We have \[
M_{s}\cdot I(x_{s})\leq I_{1}\leq M_{s}\cdot(I(x_{s})+k)\]
since we may gain at most $M_{s}k$ occurrences at the edges of the
$x_{s}$'s but we can't lose occurrences. Also we have the trivial
bound $I_{2}\leq\ell(y_{s})$. Therefore\[
M_{s}I(x_{s})\leq I(x_{s+1})\leq M_{s}(I(x_{s})+k)+\ell(y_{s})\]
and substituting this and $L_{s+1}=M_{s}L_{s}+\ell(y_{s})$ into inequality
\ref{eq:density-estimate} we get\[
\frac{M_{s}\cdot I(x_{s})}{M_{s}L_{s}+\ell(y_{s})}\leq p(L_{s+1})\leq\frac{M_{s}\cdot I(x_{s})+\ell(y_{s})+(M_{s}+1)k}{M_{s}L_{s}+\ell(y_{s})}\]
 dividing the middle term by $p(L_{s})$ and using (\ref{eq:density-estimate})
again, we get\[
\frac{1}{1+k/I(x_{s})}\cdot\frac{1}{1+\ell(y)/M_{s}L_{s}}\leq\frac{p(L_{s+1})}{p(L_{s})}\leq\frac{1+k/I(x_{s})+(\ell(y)+k)/M_{s}I(x_{s})}{1+\ell(y)/M_{s}L_{s}}\]
We saw above that $k/I(x_{s})$ is exponentially small in $s$. Thus
if $\{M_{n}\}$ grows quickly enough, both the expression on the left,
which we denote $\alpha_{s}$, and the expression on the right, which
we denote $\beta_{s}$, converge to $1$ rapidly enough for their
product to converge to a finite positive number. Now the relation
$\alpha_{s}\leq\frac{p(L_{s+1})}{p(L_{s})}\leq\beta_{s}$ and the
fact that $0<\prod_{r}^{\infty}\alpha_{s},\prod_{r}^{\infty}\beta_{s}<\infty$
implies $p(L_{s})$ converges to a positive limit as $s\rightarrow\infty$.\end{proof}
\begin{claim}
For $U$ and $p(n)$ as above, $\lim_{n\rightarrow\infty}p(n)$ exists
and is positive if $p(n)>0$ for some $n$.\end{claim}
\begin{proof}
Let $p=\lim p(L_{s})$, the limit of $p(n)$ along the subsequence
$L_{s}$. To show that $p(n)\rightarrow p$, we show that if $L_{s}\leq n<L_{s+1}$
then $p(n)/p(L_{s-1})$ is close to $1$, in a manner depending on
$s$ and tending to $1$ with $s$. To see this, recall that\begin{eqnarray*}
x_{s+1} & = & (x_{s}x_{s}\ldots x_{s})y_{s}\\
 & = & \left((x_{s-1}\ldots x_{s-1}y_{s-1})\ldots(x_{s-1}\ldots x_{s-1}y_{s-1})\right)y_{s}\end{eqnarray*}
Write $a_{n}$ for the front $n$-segment of $x_{s+1}$. Then there
is a unique way to write $a_{n}$ as\[
a_{n}=(x_{s}\ldots x_{s})(x_{s-1}\ldots x_{s-1})w\]
with $w$ a front segment of either $x_{s-1},y_{s-1}$ or $y_{s}$. 

For $n\geq L_{s}$ the number of $x_{s}$'s appearing is at least
$1$. Now consider two alternatives: If $w$ is a front segment of
either $x_{s-1}$ or $y_{s-1}$ then $\ell(w)$ is negligible compared
to $\ell(a_{n})$ because $\ell(a_{n})\geq\ell(x_{s})\geq M_{s-1}\ell(x_{s-1})$
and $M_{s-1}$ has been chosen large. On the other hand if $w=y_{s}$
then all $M_{s}$ repetitions of $x_{s}$ appear in $a_{n}$, and
again we have that $\ell(w)$ is negligible compared to $\ell(a_{n})$. 

An estimate like the one carried out for $p(L_{s})$ shows that we
can ignore edge effects and write $p(n)$ as some weighted average
of $p(L_{s})$ and $p(L_{s-1})$. But we know already that $p(L_{s})/p(L_{s-1})\rightarrow1$,
so $p(n)\approx p(L_{s-1})\rightarrow p$. 
\end{proof}

\subsection{The only ergodic measures on $X$ are $\mu$ and the point mass $\delta_{\overline{0}}$ }

A-priori the measure $\mu$ for which $x_{*}$ is generic need not
be ergodic. Rather than prove directly that $\mu$ is ergodic, we
will show that if $\nu$ is any ergodic measure on $(X,\sigma)$ then
$\nu$ is a convex combination of $\mu$ and $\delta_{\overline{0}}$.
This implies that $\mu$ is an extreme point of the convex set of
invariant measures on $X$, so it is ergodic and is the only ergodic
measure on $X$ other then $\delta_{\overline{0}}$.
\begin{thm}
The only ergodic measures for $(X,\sigma)$ are $\mu$ and $\delta_{0}$.\end{thm}
\begin{proof}
Using lemma \ref{lemma:structure-of-x-star}, we can select a sequence
$r(n)\rightarrow\infty$ and write \[
x_{*}=b_{1,n}b_{2,n}b_{3,n}\ldots\]
such that each $b_{i,n}$ is either equal to $x_{r(n)}$, or has the
property that $\ell(b_{i,n})\geq3^{L_{r(n)}}$ and all but the final
$2L_{r(n)}$ coordinates are $<1/n$. 

If $\nu$ is an ergodic measure for $(X,\sigma)$ then for some sequence
with $m(n)-k(n)\rightarrow\infty$ we have \[
\nu=\lim_{n\rightarrow\infty}\frac{1}{m(n)-k(n)+1}\sum_{i=k(n)}^{m(n)}\delta_{\sigma^{i}x_{*}}\]
(this follows from the fact that by the ergodic theorem $\nu$ has
generic points, and these can be approximated arbitrarily well by
shifts $\sigma^{i}(x_{*})$ of $x_{*}$). By passing to sub-sequences
we can assume that $m(n)-k(n)>2^{L_{r(n)}}$; denote $w_{n}=x_{*}|_{[k(n),m(n)]}$
so that $\ell(w_{n})>2^{L_{r(n)}}$. Write $\lambda_{n}$ for the
total number of indices $i=1,\ldots,\ell(w_{n})$ such that $i$ is
in a word $b_{j,n}$ with $b_{j,n}=x_{r(n)}$. We may further assume,
by passing to a subsequence, that $\lambda_{n}\rightarrow\lambda\in[0,1]$. 

Now we can write $w_{n}=b'b_{i(n),n}\ldots b_{j(n),n}b''$ for some
$i(n)<j(n)$ and $b',b''$ as short as possible. Notice that if $b_{i(n)-1,n}$
or $b_{j(n)+1,n}$ are $x_{r(n)}$ then their lengths, respectively,
are negligible (logarithmic) compared to $\ell(w_{n})$, and so also
are the lengths of $b',b''$, respectively. On the other hand, if
$b_{i(n)-1,n}$ is not $x_{r(n)}$ and if the length of $b'$ is more
than $\frac{1}{n}\ell(w_{n})$, then that word is made up almost entirely
of coordinates of magnitude less than $1/n$. Similar reasoning holds
for $b''$. It is now simple to verify the following:
\begin{itemize}
\item If $\lambda=0$ then for large $n$ most of $w_{n}$ is made up of
coordinates of magnitude $<1/n$, so in this case we have $\nu=\delta_{0}$. 
\item If $\lambda=1$, then for large $n$, the distribution of words of
length $\sqrt{L_{r(n)}}$ in $w_{n}$ is very close to their distribution
in $x_{r(n)}$, and since $r(n)\rightarrow\infty$ we have $\nu=\mu$
in this case.
\item Finally for $0<\lambda<1$ the same reasoning as above shows that
\[
\nu=\lambda\mu+(1-\lambda)\delta_{0}\]
 (note that because the lengths of the $b_{i,n}$ tend to infinity
with $n$, the statistics of subwords of $w_{n}$ of length $\sqrt{L_{r(n)}}$
are only very slightly affected by the places where two $b_{i,n}$'s
meet. Since we assumed that $\nu$ is ergodic, this is impossible.
\end{itemize}
Thus $\nu=\delta_{0}$ or $\nu=\mu$. Since $\mu\neq\delta_{0}$ this
implies that $\mu$ is ergodic. This completes the proof.
\end{proof}

\subsection{\label{sub:Further-comments}Further comments}

This example is optimal in the following sense. Any minimal system
$(X,T)$ has the property that on some dense $G_{\delta}$ subset
of $X$ the preimage of any point is a single point. Thus there are
no minimal extremely non-invertible systems. Thus if we want an extremely
non-invertible system supporting a global ergodic measure we cannot
hope for a uniquely ergodic example. The example we have given is
the next best thing: it has only two invariant measures and a unique
minimal subsystem, the fixed point $\overline{0}$.

The construction can be modified in several ways. For distance one
can guarantee that the preimage set of every point is large: by augmenting
the two functions $\theta_{0},\theta_{1}$ at each stage of the construction
with other functions it is not hard to make the preimage set of every
point of cardinality $2^{\aleph_{0}}$. By modifying $\theta_{0},\theta_{1}$
in a more complex way one can replace the minimal subsystem $\{\overline{0}\}$
with other systems.

Using the last modification, one can establish Example \ref{exa:non-invertible-example-2}
by taking the product of the resulting system with the one-sided two-shift
$\{0,1\}^{\mathbb{N}}$, and the product (Bernoulli) measure. This
yields a system with infinitely many preimages for every point, no
small pre-images, and a globally supported ergodic measure of entropy
$\log2$. In this example there are many other invariant measures;
by a more careful choice of the system we multiply with, e.g. a minimal,
uniquely ergodic subshift with a weak mixing invariant measure of
entropy $\log2$, this can be avoided.

Finally, in the construction we defined words $w_{b,k}=\tau_{b}(w_{k})$
where $b$ varies over all $0,1$-valued sequences of a fixed length.
By varying this length in a {}``random'' way the measure $\mu$
can be made to be weakly mixing, and perhaps even strongly mixing.

\bibliographystyle{plain}
\bibliography{topological-dynamics-ergodic-theory,topological-entropy,bib}

\end{document}